\documentclass{amsart}
\usepackage{amsfonts}
\usepackage{color}

\usepackage{graphicx}
\usepackage{enumerate}

\usepackage{color}   
\usepackage{hyperref}
\hypersetup{
    colorlinks=true, 
    linktoc=all,     
    linkcolor=blue,  
}

\newtheorem{thm}{Theorem}[section]
\newtheorem{cor}[thm]{Corollary}
\newtheorem{lema}[thm]{Lemma}
\newtheorem{prop}[thm]{Proposition}

\theoremstyle{definition}
\newtheorem{defn}[thm]{Definition}
\newtheorem{exam}[thm]{Example}

\theoremstyle{remark}
\newtheorem{rem}[thm]{Remark}

\def\glim{\mathop{\text{\normalfont $\Gamma-$lim}}}
\def\supp{\mathop{\text{\normalfont supp}}}
\def\diver{\mathop{\text{\normalfont div}}}
\def\diam{\mathop{\text{\normalfont diam}}}

\numberwithin{equation}{section}
\newcommand{\R}{\mathbb R}
\newcommand{\Sn}{{\mathbb S}^{n-1}}

\newcommand{\N}{\mathbb N}

\newcommand{\g}{\mathfrak g}

\def\C{\mathbf {C}}
\def\J{{\mathcal{J}}}
\def\F{{\mathcal{F}}}

\newcommand{\intr}{\int_{\R^n}}

\newcommand{\ve}{\varepsilon}

\newcommand{\cd}{\rightharpoonup}

\begin{document}
\title{Fractional order Orlicz-Sobolev spaces}

\author[J. Fern\'andez Bonder and A.M. Salort]{Juli\'an Fern\'andez Bonder and Ariel M. Salort}

\address{Departamento de Matem\'atica, FCEyN - Universidad de Buenos Aires and
\hfill\break \indent IMAS - CONICET
\hfill\break \indent Ciudad Universitaria, Pabell\'on I (1428) Av. Cantilo s/n. \hfill\break \indent Buenos Aires, Argentina.}

\email[J. Fern\'andez Bonder]{jfbonder@dm.uba.ar}
\urladdr{http://mate.dm.uba.ar/~jfbonder}

\email[A.M. Salort]{asalort@dm.uba.ar}
\urladdr{http://mate.dm.uba.ar/~asalort}




\begin{abstract}
In this paper we define the fractional order Orlicz-Sobolev spaces, and prove its convergence to the classical Orlicz-Sobolev spaces when the fractional parameter $s\uparrow 1$ in the spirit of the celebrated result of Bourgain-Brezis-Mironescu. We then deduce some consequences such as $\Gamma-$convergence of the modulars and convergence of solutions for some fractional versions of the $\Delta_g$ operator as the fractional parameter $s\uparrow 1$.
\end{abstract}

\maketitle
\tableofcontents

\section{Introduction}

In recent years there has been an increasing attention on problems involving anomalous diffusion due to some new interesting application in the natural sciences. Just to cite a few examples see the articles \cite{DGLZ, Eringen, Giacomin-Lebowitz, Laskin, Metzler-Klafter, Zhou-Du} for some physical models, \cite{Akgiray-Booth, Levendorski, Schoutens} for applications in finance, \cite{Constantin, Dalibard} for fluid dynamics, \cite{Humphries, Massaccesi-Valdinoci, Reynolds-Rhodes} for some examples in ecology and \cite{Gilboa-Osher} for image processing.

Up to date is almost impossible to give a comprehensive list of references and we refer the interested reader, for instance,  to the surveys \cite{DPV,Mosconi-Squassina,RosOton}.

In most of these applications a fundamental tool to treat these type of problems is the so-called {\em fractional order Sobolev spaces} that for $0<s<1\le p<\infty$ are defined as
$$
W^{s,p}(\R^n) = \left\{u\in L^p(\R^n)\colon \frac{u(x)-u(y)}{|x-y|^{\frac{n}{p}+s}}\in L^p(\R^n\times\R^n)\right\}.
$$

The connection of these spaces with the classical Sobolev spaces in the Hilbert case (i.e. $p=2$) is well known since the 60s  using that we have at hand the Fourier transform and one can prove the alternative characterization
$$
W^{s,2}(\R^n) = H^s(\R^n) = \{u\in L^2(\R^n)\colon (1+|\xi|^{2s}){\mathcal F}[u](\xi)\in L^2(\R^n)\}.
$$
From this characterization is then easy to show that, in a suitable sense, $H^s(\R^n)\to H^1(\R^n)$ when $s\uparrow 1$. See \cite{Stein} and also \cite{DPV}.

The general problem was tackled by J. Bourgain, H. Brezis and P. Mironescu in a series of papers \cite{BBM2, BBM} (see also  Maz'ya-Shaposhnikova \cite{MaSha} for the case $s\downarrow 0$). In particular, in \cite{BBM}, the authors prove that for any $u\in L^p(\R^n)$, it holds that
$$
\lim_{s\uparrow 1} (1-s)\iint_{\R^n\times\R^n} \frac{|u(x)-u(y)|^p}{|x-y|^{n+sp}}\, dxdy = K(n,p)\int_{\R^n} |\nabla u|^p\, dx,
$$
where $K(n,p)$ is a (explicit) constant depending only on $n$ and $p$. The right hand side of the former inequality is understood as $\infty$ if $u\not\in W^{1,p}(\R^n)$. So in that sense $W^{s,p}(\R^n) \to W^{1,p}(\R^n)$ as $s\uparrow 1$.

Let us point out that recently there has been some generalizations of these results for the so-called {\em magnetic spaces}. See \cite{SquassinaVolzone}.

On the other hand, in many contexts it is useful to consider asymptotic behaviors different that power laws, or different behaviors near 0 and near $\infty$. See \cite{Lieberman91}.

In that contexts the mathematical tool commonly used to deal with those problems is to replace the Lebesgue and Sobolev spaces for the Orlicz and Orlicz-Sobolev spaces. That is, given $G\colon \R\to\R$ an Orlicz function (c.f. Section \ref{sec.Orlicz} for precise definitions), we consider the spaces
$$
L^G(\R^n) = \{u\in L^1_{\text{loc}}(\R^n)\colon \int_{\R^n} G(|u|)\, dx <\infty\}
$$
and
$$
W^{1,G}(\R^n) = \{u\in W^{1,1}_{\text{loc}}(\R^n)\colon u, |\nabla u|\in L^G(\R^n)\}.
$$

These spaces have been extensively studied since the 50s and is by now a well understood tool to deal with nonstandard growth problems. See \cite{KrRu61, Lieberman91, SandraNoemi}.

So in this paper we propose what we believe is the {\em natural fractional version} of these spaces, i.e.
$$
W^{s,G}(\R^n) = \left\{u\in L^G(\R^n)\colon \iint_{\R^n\times\R^n} G\left(\frac{|u(x)-u(y)|}{|x-y|^s}\right) \frac{dxdy}{|x-y|^n}<\infty \right\}.
$$
Observe that in the case $G(t)=t^p$, these spaces coincide with the fractional order Sobolev space $W^{s,p}(\R^n)$.

We begin this paper by reviewing some natural properties of the spaces $W^{s,G}(\R^n)$ that are immediately deduced from the general theory of Orlicz spaces and after that we arrive at the main point of the article, i.e. to study the limit of these spaces as $s\uparrow 1$.

We follow the approach of J. Bourgain, H. Brezis and P. Mironescu in \cite{BBM} and show that there exists an Orlicz function $\tilde G$, that is computed explicitly in terms of $G$ and that is equivalent to $G$, such that $W^{s,G}(\R^n)\to W^{1,\tilde G}(\R^n)$ when $s\uparrow 1$ in the same sense that in the classical fractional Sobolev spaces case.

In some parts of the proofs we also benefit from arguments found in the article of A. Ponce \cite{Ponce04}.

We want to remark that in \cite{Ponce04} the author found some sort of nonlocal approximations to Orlicz-Sobolev spaces, but his approximations do not define fractional versions of these spaces nor are generalization of fractional order Sobolev spaces.

Finally, we end this paper with some applications of our results to the existence and stability theory for solutions to some nonlocal problems with nonstandard growth.

\section{General properties for Orlicz functions}\label{sec.Orlicz}

In this section we will define the fractional order Orlicz-Sobolev spaces and study some basic properties.

\subsection{Orlicz functions}

We start by recalling the definition of the well-known Orlicz functions.
\begin{defn}\label{defi.Orlicz}
$G\colon \R_+ \to \R_+$ is called an \emph{Orlicz function} if it has the following properties:
\begin{align}
\tag{$H_1$}\label{H1} &G \text{ is continuous, convex,  increasing and }  G(0)=0.\\ 
\tag{$H_2$}\label{H2} &G \text{ satisfies the $\Delta_2$ condition, that is there exists $\C>2$ such that }\qquad\qquad\\
\notag &G(2x)\leq \C G(x) \quad \text{ for all } x\in \R_+.\\
\tag{$H_3$}\label{H3} &G \text{ is super-linear at zero, that is } \lim_{x\to 0} \frac{G(x)}{x} = 0.
\end{align}
\end{defn}

\begin{exam}\label{ej.orlicz}
Some examples of functions verifying Definition \ref{defi.Orlicz} include the most common appearances in the literature. For instance:
\begin{enumerate}
\item $G(t) = t^p$ with $p>1$.

\item $G(t)=t^p(|\log t| + 1)$ with $p>1$.

\item If $G_1$ and $G_2$ are Orlicz functions, then $G_1\circ G_2$ is also an Orlicz function.

\item If $G_1,\dots,G_m$ are Orlicz functions and $a_1,\dots,a_m\ge 0$, then $G=\sum_{i=1}^m a_i G_i$ is an Orlicz function.

\item If $G_1,\dots,G_m$ are Orlicz functions, then $G=\max\{G_1,\dots,G_m\}$ is an Orlicz function.
\end{enumerate}

See \cite{KrRu61} for a proof of these facts and for more examples of Orlicz functions.
\end{exam}

\begin{rem}
Without loss of generality $G$ can be normalized such that $G(1)=1$.
\end{rem}

It is easy to check that  Orlicz functions fulfill the following basic properties.

\begin{lema}
Let $G:\R_+\to\R_+$ be an Orlicz function. It follows that
\begin{align}
\tag{$P_1$}\label{P1} &\text{\em (Regularity) } G \text{ is  Lipschitz continuous.}\\
\tag{$P_2$}\label{P2} &\text{\em (Integrability near $0$ and  infinity)} \text{ Given } s\in (0,1),\\
\notag & \int_1^\infty  \frac{G(x^{-s})}{x} \,dx \le \frac{\g}{s}, \quad \int_0^1 \frac{G(x^{1-s})}{x}\,dx \leq  \frac{\g}{1-s},\\
\notag & \text{where } \g:= \sup_{x\in (0,1)} x^{-1}G(x).\\
\tag{$P_3$}\label{P3} & G \text{ can be represented in the form}\\
\notag &G(x)=\int_0^x g(s)\,ds,\\
\notag & \text{where $g(s)$ is a non-decreasing right continuous function.}\qquad \qquad\qquad\qquad\\
\tag{$P_4$}\label{P4} &\text{\em (Subaditivity)} \text{ Given }  a,b\in \R_+,\\
\notag & G(a+b)\leq \tfrac{\C}{2} (G(a)+G(b)),\\
\notag & \text{where $\C>0$ is the constant in the $\Delta_2$ condition \eqref{H2}.}\\
\tag{$P_5$}\label{P5} &\text{ For any $0<b<1$ and $a>0$, it holds  $G(ab)\leq b G(a)$.}
\end{align}
\end{lema}

\begin{proof}
Properties \eqref{P1} and \eqref{P5} are direct consequences of the convexity of $G$. 

Now, \eqref{P2} is immediate once one observe that $\g<\infty$ by \eqref{H3}.

Property \eqref{P3} is proved in \cite[Theorem 1.1]{KrRu61}. Finally, \eqref{P4} follows from the convexity of $G$ together with the $\Delta_2$ condition:
$$
	G(a+b)=G(\tfrac12(2a+2b)) \leq \tfrac12(G(2a)+G(2b)) \leq \tfrac{\C}{2}(G(a)+G(b)).
$$
The proof is complete.
\end{proof}

It is shown in \cite[Theorem 4.1]{KrRu61} that the $\Delta_2$ condition \eqref{H2} is equivalent to
\begin{equation} \label{H4.1}
\frac{G'(a)}{G(a)}\leq \frac{p}{a}, \qquad \forall a>0,
\end{equation}
for some $p>1$.

The following lemma will be useful in the sequel.
\begin{lema} \label{lema.lieber}
Let $G$ be an Orlicz function. Then, for every $a\ge 1$ and $b\ge 0$ it holds
\begin{equation}\label{H4}
G(ab)\leq a^p G(b),
\end{equation}
where $p>1$ is given by \eqref{H4.1}.
\end{lema}

\begin{proof}
Let us define the function $h(t)=t^{-p}G(t)$ and observe that \eqref{H4.1} implies that
\begin{align*}
h'(t)&= t^{-p} G'(t) -p t^{-p-1} G(t)\\
&\leq t^{-p-1} ( t G'(t) -p G(t)) \leq 0.
\end{align*}
Hence, since $a\geq 1$, the lemma follows from the inequality $h(ab)\leq h(b)$.
\end{proof}

Using Lemma \ref{lema.lieber}, one obtains the following replacement for the triangle inequality for Orlicz functions.
\begin{lema}  \label{lema.triang}
Let $G$ be a Orlicz function. Then for every $\delta>0$, there exists $C_\delta>0$ such that
$$
	G(a+b) \leq C_\delta G(a)+(1+\delta)^p G(b), \qquad a,b\geq 0.
$$
\end{lema}
\begin{proof}
Let $a,b\geq 0$ and $\delta>0$. If $b>\delta a$, from the monotonicity of $G$ and \eqref{H2} we get
$$
G(a+b)\leq G(b(1+\tfrac{1}{\delta})) \leq G(b2^\kappa) \leq \C^{\kappa} G(b),
$$
where $\kappa = \kappa(\delta)\in\N$ is such that $1+\tfrac{1}{\delta}\leq 2^{\kappa}$.

Now, if $b\le \delta a$ we get from \eqref{H4},
$$
G(a+b)=G(a(1+\delta)) \leq (1+\delta)^p G(a).
$$
The proof is finished.
\end{proof}

Next, we prove a technical lemma that can be seen as the counterpart of \eqref{P5}. It will be useful in proving some properties on Orlicz spaces.

\begin{lema} \label{lema.iterado}
Let $G$ be an Orlicz function. Then, there exists $q>1$ such that 
\begin{equation} \label{ec.escala}
t^{2q} G(a)\leq G(at),
\end{equation}
for every $a>0$ and $0\le t\le 1$.
\end{lema}

\begin{proof}
Given $t\in (0,1)$, there exists $k\in\N_0$ such that $2^{-(k+1)}\le t < 2^{-k}$. So
$$
\C^{-(k+1)}\le t^q < \C^{-k}, 
$$
where $q=\tfrac{\log \C}{\log 2}>1$ since the doubling constant $\C$ is greater than $2$. From this inequality it readily follows that
$$
\C^{k+1}\le t^{-\frac{k+1}{k}q}\le t^{-2q},
$$
since $0<t<1$.

Therefore, using \eqref{H2}, we obtain
$$
G(a) \leq \C^{k+1} G(a 2^{-(k+1)}) = t^{-2q} G(a t)
$$
and the lemma is proved.
\end{proof}

\begin{rem} \label{remark.cota.inf}
Observe that, from the convexity of $G$ and since $G(0)=0$, $G(1)=1$, it follows that $G(a)\ge a$ for any $a\ge 1$.

Therefore we get a lower bound for $G$ of the form
$$
\min\{a,a^{2q}\} \leq G(a),
$$
for any $a\ge 0$.
\end{rem}

To end this subsection, we recall some tools from convex analysis. In fact, given $G$ an Orlicz function, we define the complementary function $G^*$ as
\begin{equation}\label{G*}
G^*(a) = \sup\{at-G(t)\colon t>0\}.
\end{equation}
From \eqref{G*} is immediate that the following Young-type inequality holds
\begin{equation}\label{young}
at\le G(t) + G^*(a)\quad \text{for every } a,t\ge 0.
\end{equation}

The following property will be useful in the sequel.
\begin{lema}\label{G*g}
Let $G$ be an Orlicz function and $G^*$ its complementary function. Then
$$
G^*(g(t))\le (p-1)G(t),
$$
where $p>1$ is given by \eqref{H4.1} and $g=G'$.
\end{lema}

\begin{proof}
Let $h(a):=g(t)a-G(a)$. Then, is immediate to see that $h(a)\le h(t)$ for every $a>0$. This is equivalent to say that
\begin{equation}\label{caso.igualdad}
G^*(g(t)) = g(t)t-G(t).
\end{equation}
Combining this identity with \eqref{H4.1} we obtain the result.
\end{proof}

\subsection{Fractional Orlicz--Sobolev spaces}
Given an Orlicz function $G$ and a fractional parameter $0<s\le 1$,  we consider the spaces $L^G(\R^n)$ and $W^{s,G}(\R^n)$ defined as 
\begin{align*}
&L^G(\R^n) :=\left\{ u\colon \R^n \to \R \text{ measurable, such that }  \Phi_{G}(u) < \infty \right\},\\
&W^{s,G}(\R^n):=\left\{ u\in L^G(\R^n) \text{ such that } \Phi_{s,G}(u)<\infty \right\},
\end{align*}
where the modulars $\Phi_G$ and $\Phi_{s,G}$ are defined as
\begin{align*}
\Phi_{G}(u)&=\int_{\R^n} G(|u(x)|)\,dx,\\
\Phi_{s,G}(u)&=
\begin{cases}
\displaystyle{  \iint_{\R^n\times\R^n} G\left( \frac{|u(x)-u(y)|}{|x-y|^s} \right) \frac{ dx\,dy}{|x-y|^n}} &\qquad  \text{ if } 0<s<1\\
\displaystyle{   \int_{\R^n} G(|\nabla u(x)|)\,dx }&\qquad  \text{ if } s=1.
\end{cases}
\end{align*}

These spaces are endowed with the so-called Luxemburg norms that are defined as
$$
\|u\|_G =  \|u\|_{L^G(\R^n)} := \inf\left\{\lambda>0\colon \Phi_G\left(\frac{u}{\lambda}\right)\le 1\right\}
$$
and
$$
\|u\|_{s,G} =  \|u\|_{W^{s,G}(\R^n)} := \|u\|_G + [u]_{s,G},
$$
where
$$
[u]_{s,G} :=\inf\left\{\lambda>0\colon \Phi_{s,G}\left(\frac{u}{\lambda}\right)\le 1\right\}.
$$
For $0<s<1$, the term $[\, \cdot\,]_{s,G}$ will be called the {\em $(s,G)-$Gagliardo seminorm}.

We begin by recalling some well known properties of the spaces $L^G(\R^n)$ and $W^{1,G}(\R^n)$. For a comprehensive study of these spaces, we refer to the monograph \cite{Adams} where more general Orlicz functions are considered.

\begin{prop}[\cite{Adams}, Chapter 8] \label{propiedades}
Let $G$ be an Orlicz function according to Definition \ref{defi.Orlicz}.

Then the spaces $L^G(\R^n)$ and $W^{1,G}(\R^n)$ are reflexive, separable Banach spaces. Moreover, the dual space of $L^G(\R^n)$ can be identified with $L^{G^*}(\R^n)$. Finally, $C_c^\infty(\R^n)$ is dense in $L^{G}(\R^n)$ and in $W^{1,G}(\R^n)$.
\end{prop}

It is fairly straightforward to see that the same functional properties hold for the fractional spaces $W^{s,G}(\R^n)$. We state the result for further references and make a sketch of the proof for completeness.
\begin{prop}\label{prop.WsG}
Let $G$ be an Orlicz function according to Definition \ref{defi.Orlicz} and let $0<s<1$ be a fractional parameter. 

Then $W^{s,G}(\R^n)$ is a reflexive and separable Banach space.  Moreover, $C^\infty_c(\R^n)$ is dense in $W^{s,G}(\R^n)$.
\end{prop}

\begin{proof}
First observe that if we denote $d\mu = |x-y|^{-n} \, dxdy$, then $d\mu$ is a regular Borel measure on the set $\Omega\times\Omega$ and the space $L^G(d\mu)$ is also a reflexive and separable Banach space.

Next, observe that the map
\begin{align*}
\Psi\colon W^{s,G}(\R^n) \to L^G(\R^n) \times L^G(d\mu),\quad
u \mapsto \left(u, \frac{u(x)-u(y)}{|x-y|^s}\right)
\end{align*}
is an isometry.

Therefore, the reflexivity and separability properties of $W^{s,G}(\R^n)$ are deduced from the ones of $L^G$.

Finally, the density result follows by the usual argument of truncation and regularization by convolution and uses Jensen's inequality. The details are analogous to that of the proof in the $L^G$ case and are left to the reader.
\end{proof}

\begin{rem}
Let $G$ be an Orlicz function according to Definition \ref{defi.Orlicz}. Observe that if we denote by $W^{-s,G^*}(\R^n)$ the (topological) dual space of $W^{s,G}(\R^n)$, then $L^{G^*}(\R^n)\subset W^{-s,G^*}(\R^n)$.

This is a consequence of the trivial inclusion $W^{s,G}(\R^n)\subset L^G(\R^n)$. Moreover, given $f\in L^{G^*}(\R^n)$, this inclusion is given by
$$
\langle f, u\rangle := \int_{\R^n} fu\, dx,
$$
for any $u\in W^{s,G}(\R^n)$.
\end{rem}

\subsection{Some technical lemmas}

In this subsection we analyze how the modular of a function is affected by regularization by convolution and by truncation. These facts will play a key role in the sequel.

As usual, we denote by $\rho\in C^\infty_c(\R^n)$ the standard mollifier with $\supp(\rho)=B_1(0)$ and $\rho_\ve(x)=\ve^{-n}\rho(\tfrac{x}{\ve})$ is the approximation of the identity. It follows that $\{\rho_\ve\}_{\ve>0}$ is a familiy of positive functions satisfying 
$$
\rho_\ve\in C_c^\infty(\R^n), \quad \supp(\rho_\ve)=B_\ve(0), \quad \intr \rho_\ve\,dx=1.
$$
Given $u\in L^G(\R^n)$ we define the regularized functions $u_\ve\in L^G(\R^n)\cap C^\infty(\R^n)$ as
\begin{equation} \label{regularizada}
u_\ve(x)=u*\rho_\ve(x).
\end{equation}

In this context we prove the following useful estimate on regularized functions.
\begin{lema} \label{lema.reg}
Let $u\in L^G(\R^n)$ and $\{u_\ve\}_{\ve>0}$ be the family defined in \eqref{regularizada}. Then
$$
\Phi_{s,G}(u_\ve)  \leq  \Phi_{s,G}(u)
$$
for all $\ve>0$ and $0<s<1$.
\end{lema}

\begin{proof}
By Jensen's inequality
\begin{align*}
 G\left( \left|u_\ve(x+h)-u_\ve(x)\right| |h|^{-s}  \right) &=  G\left( \left|\intr (u(x+h-y)-u(x-y)\rho_\ve(y) |h|^{-s}\,dy \right| \right)\\
 &\leq 
 \intr  G\left( \left| u(x+h-y)-u(x-y)\right| |h|^{-s} \right) \rho_\ve(y)\,dy.
\end{align*}

Integrating the last inequality over the whole $\R^n$ we get
\begin{align} \label{eq.regu}
\begin{split}
\intr G&\left( \frac{|u_\ve(x)-u_\ve(y)|}{|x-y|^s} \right) \frac{dx}{|x-y|^n}  =
\intr G\left( \frac{|u_\ve(x+h)-u_\ve(x)|}{|h|^s} \right) \frac{dx}{|h|^n} \\
&\leq 
\intr \left\{ \intr  G\left( \frac{| u(x+h-y)-u(x-y)|}{|h|^{s}} \right) \rho_\ve(y)\,dy \right\}  \frac{dx}{|h|^n} \\
&=
\intr \left\{ \intr  G\left( \frac{| u(x+h-y)-u(x-y)|}{|h|^{s}} \right) \frac{dx}{|h|^n} \right\} \rho_\ve(y)\,dy   \\
&=
\intr  G\left( \frac{| u(x+h)-u(x)|}{|h|^{s}} \right) \frac{1}{|h|^n}    \,dx
\end{split}
\end{align}
where we have used the invariance of the norm with respect to translations and the $\intr \rho \,dx =1$. Finally, since
$$
\iint_{\R^n\times\R^n} G\left( \frac{|u_\ve(x)-u_\ve(y)|}{|x-y|^s} \right) \frac{dx\,dy }{|x-y|^n} = \iint_{\R^n\times\R^n} G\left( \frac{|u_\ve(x+h)-u_\ve(x)|}{|h|^s} \right) \frac{dx\,dh }{|h|^n}
$$
the lemma follows just by integrating \eqref{eq.regu} respect to $h$.
\end{proof}

We also need estimates on modulars of truncated functions. We use the following notations: Let $\eta\in C_c^\infty(\R^n)$ such that $\eta=1$ in $B_1(0)$, $\supp (\eta)=B_2(0)$, $0\leq \eta\leq 1$ in $\R^n$ and $\|\nabla \eta\|_\infty\le 2$. Given $k\in\N$ we define $\eta_k(x)=\eta(\tfrac{x}{k})$. Observe that  $\{\eta_k\}_{k\in\N} \in C_c^\infty(\R^n)$ and
$$
0\leq \eta_k \leq 1, \quad \eta_k =1 \text{ in } B_k(0), \quad  \supp (\eta_k)=B_{2k} (0),\quad  |\nabla \eta_k|\leq \frac{2}{k}.
$$
Given $u\in L^G(\R^n)$ we define the truncated functions $u_k$, $k\in\N$ as 
\begin{equation} \label{truncada}
u_k=\eta_k u.
\end{equation}

In the next lemma we analyze the behavior of the modular of truncated functions. 

\begin{lema} \label{lema.trunc}
Let $u\in L^G(\R^n)$ and $\{u_k\}_{k\in\N}$ be the functions defined in \eqref{truncada}. Then
$$
\Phi_{s,G}(u_k) \leq \Phi_{s,G}(u) + \frac{\C^2}{2} n\omega_n\left(\frac{1}{s} + \frac{1}{k(1-s)}\right)\Phi_G (u),
$$
where $\C$ is the doubling constant defined in \eqref{H2} and $\omega_n$ is the measure of the unit ball in $\R^n$.
\end{lema}

\begin{proof}
From \eqref{P4} and since $\eta_k\leq 1$ we have
\begin{align*}
G\left(\frac{|u_k(x)-u_k(y)|}{ |x-y|^{s}}\right) &\leq  \tfrac{\C}{2}  G\left(\frac{|u(x)-u(y)|}{ |x-y|^{s}}\right) + \tfrac{\C}{2}  G\left(\frac{|u(x)||\eta_k(x)-\eta_k(y)|}{ |x-y|^{s}}\right).
\end{align*}
Then we get
\begin{align*}
\iint_{\R^n\times\R^n} G & \left(  \frac{|u_k(x)-u_k(y)|}{|x-y|^s} \right) \frac{dx\,dy}{|x-y|^n}  \leq 
\tfrac{\C}{2} \Phi_{s,G}(u)+\\
&\quad + \tfrac{\C}{2}  \iint_{\R^n\times\R^n}  G\left(\frac{|u(x)||\eta_k(x)-\eta_k(y)|}{ |x-y|^{s}}\right) \frac{dx\,dy}{|x-y|^n}.
\end{align*}

The integral above can be splited as follows.
\begin{align*}
\left(\intr \int_{|x-y|\geq 1} + \intr \int_{|x-y|< 1}\right)  G\left(\frac{|u(x)||\eta_k(x)-\eta_k(y)|}{ |x-y|^{s}}\right) \frac{dx\, dy}{|x-y|^n} := I_1 + I_2.
\end{align*}

The monotonicity of $G$ and \eqref{P5} allow us to bound $I_1$ as follows
\begin{align*}
I_1 &\leq
\intr\int_{|x-y|\geq 1} G\left(2|u(x)| \right) \frac{1}{|x-y|^{n+s}}\,dx\,dy\\
&=
n\omega_n \int_1^\infty \frac{1}{r^{s+1}}dr \intr G\left(2|u(x)| \right) \,dx\\
&\leq \frac{\C n\omega_n}{s} \intr G(|u(x)|)\,dx .
\end{align*}

We deal now with $I_2$. Observe that, since $|\nabla \eta_k|\leq \tfrac{2}{k}$ and \eqref{P5} holds, 
\begin{align*}
\int_{|x-y|<1} G\left(\frac{|u(x)||\eta_k(x)-\eta_k(y)|}{ |x-y|^{s}}\right) \frac{dx}{|x-y|^n} 
 &\leq 
\int_{|x-y|\leq 1} G\left(\frac{2}{k} \frac{|u(x)|}{  |x-y|^{s-1}}\right) \frac{dy}{|x-y|^{n}} \\
&\leq \frac{n\omega_n \C}{k}  \int_0^1 G\left(|u(x)|\right)  \, \frac{dr}{r^s}\\
&= \frac{n\omega_n \C}{k(1-s)} G\left(|u(x)|\right),
\end{align*}
where we have used \eqref{H2} in the last inequality.

From these estimates the conclusion of the lemma follows.
\end{proof}

\subsection{Asymptotic behavior}
When analyzing the asymptotic behavior of $[\,\cdot\,]_{s,G}$ when $s\uparrow 1$ one needs to understand the asymptotic behavior of some integral quantities involving the Orlicz function $G$. This is the content of this subsection.

Let us begin with the following observation.
\begin{rem} 
Given $a>0$, from the monotonicity of $G$ and \eqref{P2} it follows that
\begin{equation}\label{cota.tildeG}
\begin{split}
\int_0^1  \int_{\Sn}   G\left( a |z_n| r^{1-s}\right)   dS_z \frac{dr}{r} 
& \leq n\omega_n \int_0^1     \frac{G\left( a r^{1-s} \right)}{r}   \,dr\\
& \le n\omega_n G(a)\int_0^1 \frac{1}{r^s}\, dr = \frac{n\omega_n G(a)}{1-s},
\end{split}
\end{equation}
where $\Sn$ is the unit sphere in $\R^n$.

Then, we may define the bounded functions $\tilde G^\pm \colon \R^+ \to \R$ as
\begin{equation} \label{phitilde.mas}
\tilde G^+(a):=\limsup_{s\uparrow 1} (1-s)\int_0^1  \int_{\Sn}   G\left( a |z_n| r^{1-s}\right)   dS_z \frac{dr}{r},
\end{equation}
and $\tilde G^-(a)$  is defined analogously by changing $\limsup$ by $\liminf$.

When both $\tilde G^\pm(a)$ coincide, we define
\begin{equation} \label{phitilde}
\tilde G(a):=\lim_{s\uparrow 1} (1-s)\int_0^1  \int_{\Sn}   G\left( a |z_n| r^{1-s}\right)   dS_z \frac{dr}{r}.
\end{equation}
\end{rem}

The following proposition shows that $\tilde G^\pm$ are Orlicz functions and that they are both equivalent to $G$. Hence, the spaces $L^G(\R^n)$ and $L^{\tilde G^\pm}(\R^n)$ are, in fact, equivalents.

\begin{prop} \label{prop.equiv}
The functions $\tilde G^\pm$ defined in \eqref{phitilde.mas} are Orlicz functions. Moreover, there exist positive constants $c_1$ and $c_2$ such that
$$
c_1 G(a) \leq  \tilde G^\pm(a) \leq  c_2  G(a)\quad \text{for every } a>0.
$$
\end{prop}
\begin{proof}
We prove the proposition for $\tilde G^+$. The proof for $\tilde G^-$ is completely analogous.

First let us see that $\tilde G^+$ is an Orlicz function. Hypotheses \eqref{H1} and \eqref{H2} are trivial consequences of the fact that $G$ verifies those hypotheses.

Finally, \eqref{H3} is an immediate consequence of the facts that $G$ verifies \eqref{H3} and of \eqref{cota.tildeG}.

It remains to show that $\tilde G^+$ is equivalent to $G$. 

Observe that \eqref{cota.tildeG} gives $\tilde G^+(a)\le n\omega_n G(a)$.

On the other hand, from Lemma \ref{lema.iterado}, there exists $q>1$ such that 
\begin{align*}
\int_0^1  \int_{\Sn}   G\left(a |z_n| r^{1-s}\right)&\,dS_z \frac{dr}{r}  \geq
\int_0^1  \int_{\Sn}  (|z_n| r^{1-s})^{2q} G(a)\,dS_z \frac{dr}{r}\\
&=
G(a) \left(\int_{\Sn} |z_n|^{2q} \,dS_z \right) \left(\int_0^1   r^{2q(1-s)-1}  \,dr \right)\\
&= \frac{c(n,q)}{1-s} G(a).
\end{align*}

The result is complete.
\end{proof}

We end this subsection by computing explicitly $\tilde G$ for some of the Orlicz functions given in Example \ref{ej.orlicz}. We will denote by $K_{n,p}=\int_{\Sn} |w_n|^p\, dS_{w}$.
 
\begin{exam}
\begin{enumerate}
\item Given $p>1$, we consider  $G(t)=t^p$, $t\in \R^+$. Since
 \begin{align*}
 \int_{|w|=1}G(a|w_n| r^{1-s} )\,dS_w &= a^p r^{(1-s)p} K_{n,p} 
 \end{align*}
 we arrive at the expression
 \begin{align*}
 \tilde G(a) =  \frac{a^p}{p}  K_{n,p} .
 \end{align*}

\item Given $p>1$ we consider $G(t)=t^p|\log t|$, $t\in \R^+$. In this case we have that
 \begin{align*}
 \int_{\Sn} G(a|w_n| r^{1-s} )\,dS_w &= a^p r^{(1-s)p} \left( K_{n,p}  |\log a|   + K_{\log,n,p} - (1-s)K_{n,p} \log r\right).
 \end{align*}
 where  $K_{\log,n,p}$ is a positive constant given by
 $$
 K_{\log,n,p}=\int_{\Sn} |w_n|^p |\log |w_n| | \,dS_w.
 $$
 Therefore, we obtain
 \begin{align*}
 \tilde G(a)=  \frac{a^p}{p}\left(K_{n,p} |\log a|  + K_{\log,n,p} + \frac{K_{n,p}}{p}\right).
 \end{align*}

 \item If $G_1,\dots,G_m$ are Orlicz functions and $c_1,\dots,c_m\ge 0$, then if we denote $G=\sum_{k=1}^m c_k G_k$ we easily obtain
$$
\tilde G^\pm = \sum_{k=1}^m c_k \tilde G_k^\pm. 
$$

\item If $1<q<p$ and $G(t) = \max\{t^p, t^q\}$, then, after some computations we arrive at
$$
\tilde G(a) = \begin{cases}
\displaystyle K_{n,q} \frac{a^q}{q} & \text{if } a\le 1\\
\displaystyle 
\frac{a^q}{q} \int_{|w_n|\le \frac{1}{a}} |w_n|^q\, dS_w + \frac{a^p}{p}\int_{|w_n|> \frac{1}{a}} |w_n|^p\, dS_w \\
 + \left(\frac{1}{q}-\frac{1}{p}\right)\displaystyle\int_{|w_n|> \frac{1}{a}} \, dS_w & \text{if } a>1.
\end{cases}
$$
\end{enumerate}
\end{exam}

\section{The Rellich-Kondrachov theorem for $W^{s,G}$ spaces.}
The aim of this section is to prove the compactness of the immersion $W^{s,G}$ into $L^G$.

The spirit of the proof lies in proving an equi-continuity estimate in order to apply a variant of the well-known Fr\`echet-Kolmogorov Compactness Theorem.

The main result is this sections reads as follows.

\begin{thm} \label{teo.comp}
Let $0<s<1$ and $G$ an Orlicz function. Then for every $\{u_n\}_{n\in\N}\subset W^{s,G}(\R^n)$ a bounded sequence, i.e., $\sup_{n\in\N} ( \Phi_{s,G}(u_n) + \Phi_G(u_n) )<\infty$, there exists $u\in W^{s,G}(\R^n)$ and a subsequence $\{u_{n_k}\}_{k\in\N}\subset \{u_n\}_{n\in\N}$ such that $u_{n_k}\to u$ in $L^G_{\text{loc}}(\R^n)$.
\end{thm}

The following technical lemma provides the equi-continuity of modulars.

\begin{lema} \label{lema.comp}
Let $0<s<1$ and $G$ be an Orlicz function. Then, there exists a constant $C>0$ such that
$$
\Phi_G (\tau_h u - u) \leq C |h|^s \Phi_{s,G}(u),
$$
for every $u\in W^{s,G}(\R^n)$ and every $0<|h|<\frac12$, where $\tau_h u(x) =u(x+h)$.
\end{lema}

\begin{proof}
The $\Delta_2$ condition \eqref{H2} gives that
$$
G(| u(x+h)-u(x) |) \leq \C \left[ G(| u(x+h)-u(y) |) + G(| u(y)-u(x) |) \right]
$$
for all $y\in B_{|h|}(x)$. Then
\begin{align} \label{ec1.l1}
\begin{split}
\Phi_G (\tau_h u - u) &= \frac{1}{|B_{|h|}(x)|}\int_{B_{|h|}(x)}\intr G(| u(x+h)-u(x) |) \,dx\,dy\\
&\leq \frac{\C}{|h|^n\omega_n } \int_{B_{|h|}(x)}\intr G(| u(x+h)-u(y) |) \,dx\,dy\\
&\quad +\frac{\C}{|h|^n\omega_n }  \int_{B_{|h|}(x)}\intr G(| u(y)-u(x) |) \,dx\,dy  \\
&= \frac{\C}{|h|^n \omega_n }( I_1+I_2 ).
\end{split}
\end{align}
Given $x\in \R^n$ and $y\in B_{|h|}(x)$ we have that
$$
|x-y|\leq |h|, \qquad |x+h-y|\leq |x-y|+|h|\leq 2|h|.
$$
Then, the integral $I_1$ can be bounded as  
\begin{align*}
I_1&= \int_{B_{|h|}(x)}\intr G\left(\frac{| u(x+h)-u(y) |}{|x+h-y|^s }|x+h-y|^s\right) |x+h-y|^n \frac{dx\,dy}{|x+h-y|^n}\\
&\leq \int_{B_{|h|}(x)}\intr G\left(\frac{|u(x+h)-u(y) |}{|x+h-y|^s }2^s|h|^s\right) 2^n|h|^n \frac{dx\,dy}{|x+h-y|^n}\\
&\leq 2^{n+s}|h|^{n+s}\iint_{\R^n\times\R^n} G\left(\frac{| u(x)-u(y) |}{|x-y|^s }\right) \frac{dx\,dy}{|x-y|^n}\\
&= \C|h|^{n+s}\Phi_{s,G}(u),
\end{align*}
where we have used the monotonicity of $G$ and property \eqref{P5}.
Analogously, 
$$
I_2\leq C|h|^{n+s}\Phi_{s,G}(u).
$$
Finally, inserting the two upper bounds found above in \eqref{ec1.l1} we obtain that
$$
\Phi_G (\tau_h u - u) \leq C |h|^s \Phi_{s,G}(u)
$$
and the lemma follows.
\end{proof}

 \begin{proof}[Proof of Theorem \ref{teo.comp}]
First observe that if $\{u_k\}_{k\in \N}$ is bounded in $W^{s,G}(\R^n)$, then is also bounded in $L^G(\R^n)$. Let us denote by $M=\sup_{k\in\N} ( \Phi_{s,G}(u_k) + \Phi_G(u_k) )<\infty$. Then Lemma \ref{lema.comp} gives that
$$
\sup_{k\in \N}\Phi_G (\tau_h u_k - u_k) \leq C M |h|^s .
$$
Now, applying a variant of the Fr\`echet-Kolmogorov Theorem (see \cite[Theorem 11.5]{KrRu61}) we are able to claim the existence of a function $u\in L^G(\R^n)$ and a subsequence, that we still denote by $\{u_k\}_{k\in\N}$, such that $u_k\to u$ in $L^G_{loc}(\R^n)$.

Moreover, $u\in W^{s,G}(\R^n)$. Indeed, up to a subsequence, $u_k\to u$ a.e. $\R^n$. Then
$$
0\leq \lim_{k\to \infty}G\left( \frac{|u_k(x)-u_k(y)|}{|x-y|^s}\right) = G\left( \frac{|u(x)-u(y)|}{|x-y|^s}\right) \quad \text{a.e. } (x,y)\in\R^n\times \R^n.
$$
Therefore, Fatou's Lemma together with the lower semicontinuity of $G$ gives
\begin{align*}
\Phi_{s,G}(u) &= \iint_{\R^n\times\R^n}  G\left(\frac{|u(x)-u(y)|}{|x-y|^s} \right) \frac{dx\,dy}{|x-y|^n}\\ 
&\leq \liminf_{k\to\infty} \iint_{\R^n\times\R^n}  G\left(\frac{|u_k(x)-u_k(y)|}{|x-y|^s} \right) \frac{dx\,dy}{|x-y|^n} \\
&\leq \sup_{k\in\N}\Phi_{s,G}(u_k)\le M <\infty
\end{align*}
as required.
\end{proof}

\section{The main result}
This section is aimed at proving the natural extension to fractional Orlicz spaces of a celebrated theorem of Bourgain, Brezis and Mironescu \cite{BBM}, namely:

\begin{thm} \label{main1}
Let $G$ be an Orlicz function such that the limit in \eqref{phitilde} exists. Then, given $u\in L^{ G}(\R^n)$ and $0<s<1$ it holds that
\begin{equation} \label{bbm}
\lim_{s\uparrow 1} (1-s) \Phi_{s,G}(u) = \Phi_{\tilde G}(\nabla u),
\end{equation}
where $\tilde G$ is defined in \eqref{phitilde}.
\end{thm}

The proof of Theorem \ref{main1} will be a consequence of the following couple of lemmas.

\begin{lema} \label{teo1}
Let $u\in W^{1,G}(\R^n)$. Then, for $0<s<1$ it holds that
$$
	\Phi_{s,G}(u) \leq \frac{n\omega_n}{1-s} \Phi_{G}(|\nabla u|) + \frac{2\C n\omega_n}{s} \Phi_{G}(u)
$$
where $\C$ is the doubling constant defined in \eqref{H2}.
\end{lema}

\begin{proof}
Let us first assume that $u\in C^2_c(\R^n)$.

We split the integral
\begin{align*}
 \iint_{\R^n\times\R^n} &G\left( \frac{|u(x)-u(y)|}{|x-y|^s} \right) \frac{dx\,dy}{|x-y|^n}\\ &=
 \left(\int_{B_1}\intr + \int_{B_1^c} \intr \right)  G\left( \frac{|u(x+h)-u(x)|}{|h|^s} \right) \frac{dx\,dh}{|h|^n} \\
 &:=I_1+I_2.
\end{align*}

Let us bound $I_1$. Given $u\in C^2_c(\R^n)$, observe that for any fixed $x\in \R^n$ and $h\in\R^n$ we can write
$$
 u(x+h)-u(x) = \int_0^1 \frac{d}{dt} u(x+th) \,dt = \int_0^1 \nabla u(x+th)\cdot h \,dt.
$$
Dividing by $|h|^s$ and using the monotonicity and   convexity of $G$ we get
\begin{align} \label{ec.lema.1}
\begin{split}
 G\left( \frac{|u(x+h)-u(x)|}{|h|^s} \right) 
 &\leq 
 G \left(  \int_0^1 |\nabla u(x+th)| |h|^{1-s} \,dt \right)\\
 &\leq
    \int_0^1 G \left(|\nabla u(x+th)| |h|^{1-s} \right)dt.  
\end{split}    
\end{align}
Expression \eqref{ec.lema.1} together with \eqref{P5} allow us to bound $I_1$ as follows
\begin{align*}
I_1&\leq \int_{B_1 }\intr \int_0^1 G \left(|\nabla u(x+th)| |h|^{1-s} \right)dt\, dx \,\frac{dh}{|h|^n}\\
&\leq \int_{B_1 } |h|^{ 1-s-n}\intr \int_0^1 G \left(|\nabla u(x+th)|\right)dt\, dx \,dh\\
&\leq \int_{B_1 } |h|^{1-s-n} \,dh \intr  G \left(|\nabla u(x)|\right)  dx \\
&= n\omega_n \int_0^1 r^{-s} \,dr \intr  G \left(|\nabla u(x)|\right)  dx \\
&= \frac{n\omega_n}{1-s} \intr  G \left(|\nabla u(x)|\right)  dx .
\end{align*}

The integral $I_2$ can be bounded using \eqref{P5} (since $|h|^{-s}<1$ when $h\in B_1^c$). Indeed, 
\begin{align*}
 I_2&\leq \int_{B_1^c}h^{-s-n} \intr G(|u(x+h)-u(x)|) \,dx \,dh\\
&\leq \C \int_{B_1^c}h^{-s-n} \intr G(|u(x+h)|) + G(|u(x)|) \,dx \,dh\\ 
&= 2\C \intr G(|u(x)|)\,dx \ n\omega_n \int_1^\infty r^{-s-1} \,dr \\
&= \frac{2\C n\omega_n}{s} \intr G(|u(x)|)\,dx, 
\end{align*}
where we have used the property \eqref{P4}.

In order to prove the Lemma for any $u\in W^{1,G}(\R^n)$, we take a sequence $\{u_k\}_{k\in\N}\subset C^2_c(\R^n)$ such that $u_k\to u$ in $W^{1,G}(\R^n)$. Without loss of generality, we may assume that $u_k\to u$ a.e. in $\R^n$. Observe that this implies that
$$
G\left(\frac{|u_k(x)-u_k(y)|}{|x-y|^s}\right)\to G\left(\frac{|u(x)-u(y)|}{|x-y|^s}\right) \quad \text{a.e. in } \R^n\times\R^n.
$$

Therefore, by Proposition \ref{propiedades} and Fatou's Lemma, we obtain that
\begin{align*}
\Phi_{s,G}(u)&\le \liminf_{k\to\infty} \Phi_{s,G}(u_k) \le \lim_{k\to\infty}\left[\frac{n\omega_n}{1-s} \Phi_{G}(|\nabla u_k|) + \frac{2\C n\omega_n}{s} \Phi_{G}(u_k)\right]\\
&= \frac{n\omega_n }{1-s} \Phi_{G}(|\nabla u|) + \frac{2\C n\omega_n}{s} \Phi_{G}(u).
\end{align*} 
The proof is now complete.
\end{proof}

 The following lemma is key in the proof of the main result.
\begin{lema}
Let $G$ be an Orlicz function such that the limit in \eqref{phitilde} exists. Let $u\in C^2_c(\R^n)$. Then, for every fixed $x\in \R^n$ we have that
\begin{equation} \label{eq.bbm}
\lim_{s\uparrow 1} (1-s) \intr G\left( \frac{|u(x)-u(y)|}{|x-y|^s} \right) \frac{ dy}{|x-y|^n} =  \tilde G(|\nabla u(x)|) 
\end{equation}
where $\tilde G$ is defined in \eqref{phitilde}.
\end{lema}

\begin{proof}
For each fixed $x\in \R^n$ we split the integral 
\begin{align*}
\intr G\left(\frac{|u(x)-u(y)|}{|x-y|^s} \right) \frac{dy}{|x-y|^n}= &\int_{|x-y|<1}  G\left(\frac{|u(x)-u(y)|}{|x-y|^s} \right) \frac{dy}{|x-y|^n} \\
& +\int_{|x-y|\geq 1} G\left(\frac{|u(x)-u(y)|}{|x-y|^s} \right) \frac{dy}{|x-y|^n} \\
= & I_1 + I_2.
\end{align*}
Observe that for any $x,y\in \R^n$, $x\neq y$, property \eqref{P1} gives that
\begin{align*}
\left| G\left( \frac{|u(x)-u(y)|}{|x-y|^s} \right) - G\left( \left| \nabla u(x) \cdot \frac{x-y}{|x-y|^s}\right|  \right) \right|&\leq
L \frac{|u(x)-u(y)-\nabla u(x)\cdot (x-y)|}{|x-y|^s}\\
&\leq C |x-y|^{2-s},
\end{align*}
where $L$ is the Lipschitz constant of $G$ on the interval $[0,\|\nabla u\|_\infty]$ and $C$ depends on the $C^2-$norm of $u$.

Now, since
$$
\int_{|x-y|<1} |x-y|^{2-s-n}\, dy = \frac{n\omega_n}{2-s},
$$
it follows that
\begin{align*}
\lim_{s\uparrow 1} (1-s)\int_{|x-y|<1} &G\left( \frac{|u(x)-u(y)|}{|x-y|^s} \right) \frac{ dy}{|x-y|^n}\\
& = \lim_{s\uparrow 1} (1-s) \int_{|x-y|<1} G\left( \left| \nabla u(x) \cdot \frac{x-y}{|x-y|^s}\right|  \right)\frac{ dy}{|x-y|^n}.
\end{align*}

But
\begin{align*} 
\int_{|x-y|<1}  G\left( \left| \nabla u(x) \cdot \frac{x-y}{|x-y|^s}\right|\right) \frac{dy}{|x-y|^n} &=\int_0^1  \int_{\Sn}   G\left( \left| \nabla u(x) \cdot w\right| r^{1-s}\right)   dS_w \, \frac{dr}{r}\\
&=\int_0^1  \int_{\Sn}   G\left( \left| \nabla u(x) \right| |w_n| r^{1-s}\right)   dS_w \, \frac{dr}{r}.
\end{align*}
where we have performed a rotation such that $\nabla u(x) = |\nabla u(x)| e_n$.

Therefore
\begin{equation} \label{des.I}
\lim_{s\uparrow 1} (1-s)I_1 = \tilde G (|\nabla u(x)|).
\end{equation}

Let us deal with $I_2$. Since $G$ is increasing and \eqref{P5} holds, $I_2$ is bounded. In fact, 
\begin{align*} 
I_2 &\leq  \int_{|x-y|\geq 1}  \frac{G (2\|u\|_\infty)}{|x-y|^{n+s}} \, dy
=G (2\|u\|_\infty) n\omega_n\int_1^\infty  \frac{1}{r^{1+s}}\,dr =\frac{n\omega_n}{s}G (2\|u\|_\infty),
\end{align*}
from where we can derive that
\begin{equation} \label{des.II}
\lim_{s\uparrow 1} (1-s)I_2 = 0.
\end{equation}

Finally, from \eqref{des.I} and \eqref{des.II} we obtain \eqref{eq.bbm}.
\end{proof}

At this point we are ready to prove our main result.

\begin{proof} [Proof of Theorem \ref{main1}]
Given $u\in C_c^2(\R^n)$ with $\supp(u)\subset B_R(0)$, in view of Theorem \ref{lema.bbm} it only remains to show the existence of an integrable majorant for $(1-s)F_s$, where $F_s$ is given by
$$
	F_s(x):=\intr G\left( \frac{|u(x)-u(y)|}{|x-y|^s} \right) \frac{ dy}{|x-y|^n}.
$$
Without loss of generality we can assume that $R>1$.

First, we analyze the behavior of $F_s(x)$ for small values of $x$. When $|x|<2R$ we can write
\begin{align}
	|F_s(x)|&=\left( \int_{B_1(x)} + \int_{\R^n\setminus B_1(x)} \right) G\left( \frac{|u(x)-u(y)|}{|x-y|^s} \right) \frac{ dy}{|x-y|^n} := I_1 + I_2.
\end{align}
Arguing as in the proof of Lemma \ref{teo1} we obtain that
\begin{align} \label{dest1}
\begin{split}
I_1&\leq \int_{|h|<1} |h|^{1-s-n} \int_0^1 G(|\nabla u(x+th)|) \,dt \, dh\\
&\leq n\omega_n G(\|\nabla u\|_\infty) \int_0^1 r^{-s}\,dr\\
&=  G(\|\nabla u\|_\infty) \frac{n\omega_n}{1-s},
\end{split}
\end{align}
and
\begin{align} \label{dest2}
\begin{split}
I_2&\leq  \int_{|h|\geq 1} |h|^{-s-n} (G(|u(x+h)|+|u(x)|)) \, dh\\
&\leq n\omega_n \int_1^\infty r^{-s-1} G(2\|u\|_\infty)\, dr\\
&= \frac{n\omega_n }{s}G(2\|u\|_\infty).
\end{split}
\end{align}

When $|x|\geq 2R$ the function $u$ vanishes and we have that
$$
	F_s(x)= \int_{B_R(0)} G\left( \frac{|u(y)|}{|x-y|^s} \right) \frac{ dy}{|x-y|^n}.
$$
Since $|x-y| \geq |x|-R\geq \frac12|x|$, from the monotonicity of $G$, $\Delta_2$ condition and \eqref{P5} (since $|x|\geq 2$) we get 
\begin{align} \label{dest3}
\begin{split}
	|F_s(x)|&\leq \frac{2^n}{|x|^n}\int_{B_R(0)} G\left( \frac{2^s|u(y)|}{|x|^s} \right) \,dy \leq 
	\frac{\C}{|x|^n}\int_{B_R(0)} G\left( \frac{|u(y)|}{|x|^s} \right) \,dy\\
	&\leq 
	\frac{\C}{|x|^{n+s}}\int_{B_R(0)} G\left( |u(y)| \right) \,dy	
	\leq 
	\frac{\C}{|x|^{n+\frac12 }}\int_{B_R(0)} G\left( |u(y)| \right) \,dy,
\end{split}	
\end{align}
for any $s\ge \frac12$.

From \eqref{dest1}, \eqref{dest2} and \eqref{dest3} we obtain that
$$
(1-s)|F_s(x)|\leq   C\left(\chi_{B_R(0)}(x) + \frac{1}{|x|^{n+\frac12}} \chi_{B_R(0)^c}(x) \right) \in L^1(\R^n)
$$
  where $  C>0$ depends on $n$, $p$ and $u$ but it is independent of $s$.

Then, from Theorem \ref{lema.bbm} and the Dominated Convergence Theorem the result follows for any $u\in C_c^2(\R^n)$.

Let us extend the result for any $u\in W^{1,G}(\R^n)$. According to Proposition \ref{propiedades}, let $\{u_k\}_{k\in\N}\subset C_c^2(\R^n)$ be a sequence such that $u_k\to u$ in $W^{1,G}(\R^n)$. Then
\begin{align} \label{desig00}
\begin{split}
\left| (1-s)\Phi_{s,G}(u) - \Phi_{\tilde G}(\nabla u) \right| &\leq  (1-s) \left| \Phi_{s,G}(u) - \Phi_{s,G}(u_k) \right| \\
&+ \left|(1-s) \Phi_{s,G}(u_k) - \Phi_{\tilde G} (|\nabla u_k|) \right|\\
& + \left|\Phi_{\tilde G}(|\nabla u_k|) - \Phi_{\tilde G}(|\nabla u|)\right|.
\end{split}
\end{align}
Let us fix $\ve>0$. From Proposition \ref{propiedades}, there exists $k_0$ such that for $k\geq k_0$, 
$$
	|\Phi_{\tilde G}(|\nabla u_k|)-\Phi_{\tilde G}(|\nabla u|)|\leq \frac{\ve}{2},
$$
and from Lemma \ref{lema.triang} one can take $\delta>0$ (to be fixed) such that
\begin{equation} \label{desig1}
(1-s) | \Phi_{s,G}(u) -  \Phi_{s,G}(u) | \leq (1-s)\delta  \Phi_{s,G}(u_k)+ (1-s)C_\delta \Phi_{s,G}(u-u_k).
\end{equation}
Observe that from Lemma \ref{teo1} we have that $(1-s)  \Phi_{s,G}(u_k) \leq K$ for some positive constant $K$. Moreover, again from Lemma \ref{teo1}, there is some $k_1$ such that for $k\geq k_1$ it holds that $
(1-s)\Phi_{s,G}(u-u_k) \leq \frac{\ve}{4C_\delta}$. Consequently, it follows that \eqref{desig1} can be bounded as
$$
(1-s) | \Phi_{s,G}(u) -  \Phi_{s,G}(u) |   \leq \delta K + \frac{\ve}{4}
$$
for $k\geq k_1$. Hence, choosing $\delta=\frac{\ve}{4K}$ we find that \eqref{desig00} is upper bounded as
$$
\left| (1-s)\Phi_{s,G}(u) - \Phi_{\tilde G}(\nabla u) \right| \leq \ve + \left|(1-s) \Phi_{s,G}(u_k) - \Phi_{\tilde G} (|\nabla u_k|) \right|
$$
for all $k\geq \max\{k_0,k_1\}$. Finally, the desired result follows by fixing a value of $k\geq \max\{k_0,k_1\}$ and taking limit as $s\uparrow 1$.

To finish the proof, let us see that if $u\in L^G(\R^n)$ is such that
$$
\liminf_{s\uparrow 1} (1-s) \Phi_{s,G}(u)<\infty,
$$
then $u\in W^{1,G}(\R^n)$.

Given $u\in L^G(\R^n)$, according to Lemmas \ref{lema.reg} and \ref{lema.trunc}, if we define the approximating family 
$$
	u_{k,\ve}=\rho_\ve * (u\eta_k) \in C_c^\infty(\R^n),
$$
it satisfies
$$
\liminf_{s\uparrow 1} (1-s) \Phi_{s,G}(u_{k,\ve})<C,
$$
with $C$ independent on $\ve>0$ and $k\in\N$. 

The first part of this theorem gives that
$$
\Phi_{\tilde G}(\nabla u_{k,\ve}) =\lim_{s\uparrow 1} (1-s) \Phi_{s,G}(u_{k,\ve})<C,
$$
then, from Proposition \ref{prop.equiv}, $\{u_{k,\ve}\}_{k\in\N, \ve>0}$ is bounded in $W^{1,G}(\R^n)$. Consequently, from Proposition \ref{propiedades}, there exists a sequence $u_j=u_{k_j,\ve_j}$ with $k_j\to\infty$ and $\ve_k\downarrow 0$ and $\tilde u\in W^{1,G}(\R^n)$ such that $u_j\cd \tilde u$ weakly in $W^{1,G}(\R^n)$.  Moreover, since $u_{k,\ve} \to u$ in $L^G(\R^n)$ as $k\to\infty$, and $\ve\downarrow 0$, we can conclude that $\tilde u= u\in W^{1,G}(\R^n)$ as required.
\end{proof}

\section{The case of a sequence}
Our second main result in this paper is the following limit result for sequences of functions.

\begin{thm}\label{main2}
Let $G$ be an Orlicz function. Let $0\le s_k\  \uparrow 1$ and $\{u_k\}_{k\in\N}\subset L^G(\R^n)$ be such that
$$
	\sup_{k\in\N} (1-s_k)\Phi_{s_k,G}(u_k) <\infty \quad \text{ and }\quad  \sup_{k\in\N} \Phi_G(u_k) <\infty.
$$
Then there exists $u\in L^G(\R^n)$ and a subsequence $\{u_{k_j}\}_{j\in\N}\subset \{u_k\}_{k\in\N}$ such that $u_{k_j}\to u$ in $L^G_{loc}(\R^n)$. Moreover, if $G$ is such that the limit in \eqref{phitilde} exists, then $u\in W^{1,G}(\R^n)$ and the following estimate holds
$$
\Phi_{\tilde G}(\nabla u) \leq \liminf_{k\to\infty} (1-s_k) \Phi_{s_k,G}(u_k).
$$
\end{thm} 

The proof of the above result will be a direct consequence of the following useful estimate:

\begin{thm} \label{teo.s1.s2}
Let $0<s_1<s_2<1$ and $u \in L^G(\R^n)$. Then
$$
(1-s_1)\Phi_{s_1,G}(u) \leq 2^{1-s_1} (1-s_2) \Phi_{s_2,G}(u) + \frac{2\C n\omega_n(1-s_1)}{s_1} \Phi_G(u).
$$
\end{thm}

The key point in proving Theorem \ref{teo.s1.s2} is the following lemma:
\begin{lema} \label{lema.bbm}
Let $g,h:(0,1) \to \R^+$ measurable functions. Suppose that for some constant $c>0$ it holds that $g(t) \leq c g(\tfrac{t}{2})$ for $t\in (0,1)$ and that $h$ is decreasing. Then, given $r>-1$, 
$$
\int_0^1 t^r g(t)h(t)\,dt \geq \frac{r+1}{2^{r+1}}\int_0^1 t^r g(t)\,dt \int_0^1 t^r h(t) \,dt.
$$
\end{lema}
We omit the  proof Lemma \ref{lema.bbm} since it follows with a slight modification of the proof of Lemma 2 in \cite{BBM}, which is stated with $c=1$.

Now we proceed with the proof of the estimate.

\begin{proof}[Proof of Theorem \ref{teo.s1.s2}]
Given $u\in L^G(\R^n)$, we define for $t>0$ and $0<s<1$,
\begin{align*}
F_s(t)&=\int_{\Sn} \intr G\left( \frac{|u(x+tw)-u(x)|}{t^s}\right) \,dx \,dS_w\\
&=\frac{1}{t^{n-1}}\int_{|h|=t} \intr G\left( \frac{|u(x+h)-u(x)|}{|h|^s}\right) \,dx \,dS_h.
\end{align*}
From the $\Delta_2$ property and \eqref{P5} we get
\begin{align*}
G&\left( \frac{|u(x+2tw)-u(x)|}{2^st^s}\right) =
G\left( \frac{|u(x+2tw)-u(x+tw)|+|u(x+tw)-u(x)|}{2^st^s}\right)\\
&\qquad\leq 
\frac{\C}{2^s}\left\{ G\left( \frac{|u(x+2tw)-u(x+tw)|}{t^s}\right)+
G\left( \frac{|u(x+tw)-u(x)|}{t^s}\right) \right\},
\end{align*}
integrating over $\R^n$ and $\Sn$ and using the invariance of the integral we get that
$$
F_s(2t)\leq  2^{1-s} \C F_s(t).
$$
Now, if we consider the function $g_s(t)=\frac{F_s(t)}{t^{1-s}}$  we obtain that
$$
g_s(2t)=\frac{F_s(2t)}{2^{1-s} t^{1-s}}\leq  \C\frac{F_s(t)}{ t^{1-s}} = \C g_s(t).
$$
Then, observe that
\begin{align} \label{dess1}
\begin{split}
\int_{|h|<1} \intr G & \left( \frac{|u(x+h)-u(x)|}{|h|^s}\right) \,\frac{dx}{|h|^n} \,dh\\
&=\int_0^1 \int_{|h|=t} \intr G\left( \frac{|u(x+h)-u(x)|}{|h|^s}\right) \,\frac{dx}{|h|^n} \,dS_h \,dt\\
&=\int_0^1 \frac{F_s(t)}{t} \,dt=\int_0^1 \frac{g_s(t)}{t^{s}}\,dt.
\end{split}
\end{align}
Consider now $0<s_1<s_2<1$. Therefore, using \eqref{P5}, for any $0<t<1$ 
\begin{align*}
g_{s_1}(t)=\frac{F_{s_1}(t)}{t^{1-s_1}} &= \frac{1}{t^{1-s_1}} \int_{\Sn} \intr G\left( \frac{|u(x+tw)-u(x)|}{t^{s_1}}\right) \,dx \,dS_w\\
&\leq 
 \frac{1}{t^{1-s_1}} \frac{1}{t^{s_1-s_2}}\int_{S^{n-1}} \intr G\left( \frac{|u(x+tw)-u(x)|}{t^{s_2} }\right) \,dx \,dS_w\\
&= \frac{F_{s_2}(t)}{t^{1-s_2}} = g_{s_2}(t)
\end{align*}
frome where we obtain that
$$
\int_0^1 \frac{g_{s_2}(t)}{t^{s_2}} \,dt \geq 
\int_0^1 \frac{g_{s_1}(t)}{t^{s_2}} \,dt  = \int_0^1 \frac{1}{t^{s_2-s_1}}\frac{g_{s_1}(t)}{t^{s_1}} \,dt.
$$
Now, from Lemma \ref{lema.bbm} with $r=-s_1$ and $h(t)=t^{-(s_2-s_1)}$ we get
\begin{align*}
 \int_0^1 \frac{1}{t^{s_2-s_1}}\frac{g_{s_1}(t)}{t^{s_1}} \,dt &\geq \frac{1-s_1}{2^{1-s_1}} \int_0^1 t^{-s_1}g_{s_1}(t) \,dt \int_0^1 t^{-s_2}\,dt\\
 &=\frac{1}{2^{1-s_1}} \frac{1-s_1}{1-s_2}\int_0^1 \frac{g_{s_1}(t)}{t^{s_1}} \,dt
\end{align*}
and then
\begin{align} \label{dess2}
\begin{split}
 \int_0^1 \frac{g_{s_1}(t)}{t^{s_1}} \,dt &\leq 2^{1-s_1}\frac{1-s_2}{1-s_1}  \int_0^1 \frac{1}{t^{s_2-s_1}}\frac{g_{s_1}(t)}{t^{s_1}} \,dt \\
&\leq 2^{1-s_1}\frac{1-s_2}{1-s_1}  \int_0^1 \frac{g_{s_2}(t)}{t^{s_2}} \,dt.
\end{split}
\end{align}
Mixing up \eqref{dess1} and \eqref{dess2} we get that
\begin{align*}
\frac{1-s_1}{2^{1-s_1}}\int_{|h|<1} \intr G &\left( \frac{|u(x+h)-u(x)|}{|h|^{s_1}}\right) \,\frac{dx}{|h|^n} \,dS_h\\ 
&\leq 
 (1-s_2) \int_{|h|<1} \intr G\left( \frac{|u(x+h)-u(x)|}{|h|^{s_2}}\right) \,\frac{dx}{|h|^n} \,dS_h.
\end{align*}
Finally, 
\begin{align*}
\int_{|h|\geq 1} \intr G &\left( \frac{|u(x+h)-u(x)|}{|h|^{s_1}}\right) \,\frac{dx}{|h|^n} \,dS_h\\ &\leq
\C \int_{|h|\geq 1}   \intr \left(G(|u(x+h|)+G(|u(x)|)  \right)\,\frac{dx}{|h|^{s_1+n}} \,dS_h \\
&\leq 2\C n\omega_n \Phi_{G}(u)
	\int_1^\infty \frac{1}{r^{s_1+1}} \,dr \\
	&= \frac{2\C n\omega_n}{s_1} \Phi_{G}(u)
\end{align*}
and the result follows from the  last two inequalities.
\end{proof}

\begin{proof}[Proof of Theorem \ref{main2}]
 Let $0<s_k \uparrow 1$ and  $\{u_k\}_{k\in\N}\subset L^G(\R^n)$ be such that
$$
	\sup_{k\in\N} \Phi_{s_k,G}(u_k) <\infty \quad \text{ and }\quad  \sup_{k\in\N} \Phi_G(u_k):=M <\infty.
$$

Take $0<t<1$ be fixed and, from Theorem \ref{teo.s1.s2}, we obtain that $\{u_k\}_{k\in\N}\subset W^{t,G}(\R^n)$ is bounded. Hence, by Theorem \ref{teo.comp} there exists a subsequence (still denoted by $u_k$) and a function $u\in L^G(\R^n)$ such that $u_k\to u$ in $L^G_{loc}(\R^n)$. Moreover, without loss of generality, we may assume that $u_k\to u$ a.e. in $\R^n$.

From Fatou's Lemma we get
$$
	\Phi_{t,G}(u)\leq \liminf_{k\to\infty} \Phi_{t,G}(u_k),
$$
and from Theorem \ref{teo.s1.s2} we obtain that
$$
	(1-t)\Phi_{t,G}(u)\leq \liminf_{k\to\infty}  2^{1-t} (1-s_k) \Phi_{s_k,G}(u_k) + \frac{2\C n\omega_n(1-t)}{t} M.
$$
Finally, taking limit as $t\uparrow 1$ and invoking Theorem \ref{main1}, the result follows.
\end{proof}

\section{Some consequences and applications}

In this final section, we show some immediate consequences of our main theorems, Theorems \ref{main1} and \ref{main2}.

Throughout this section $G$ will be an Orlicz function such that the limit in \eqref{phitilde} exists.

When working on a domain $\Omega\subset \R^n$ (bounded or not) it is useful to introduce the following notations.

The space $W^{1,G}_0(\Omega)$ denotes, as usual, the closure of $C^\infty_c(\Omega)$ with respect to the $\|\cdot\|_{1,G}-$norm.

In the fractional setting, we use the following definitions
$$
W^{s,G}_0(\Omega) := \{u\in W^{s,G}(\R^n) \colon u=0 \text{ a.e. in } \R^n\setminus \Omega \}.
$$

Alternatively, one can consider
$$
\widetilde{W}^{s,G}(\Omega) := \overline{C_c^\infty(\Omega)}^{\|\cdot\|_{s,G}}.
$$

In the classical case, i.e. when $G(t)=t^p$, these spaces $W^{s,p}_0(\Omega)$ and $\widetilde{W}^{s,p}(\Omega)$ are known to coincide when $s<\tfrac{1}{p}$ or when $0<s<1$ and $\Omega$ has Lipschitz boundary. See \cite{DPV}.

In this paper, we shall not investigate the cases where these spaces $W^{s,G}_0(\Omega)$ and $\widetilde{W}^{s,G}(\Omega)$ coincide and use the space $W^{s,G}_0(\Omega)$ to illustrate our applications.

In what follows, every function $u\in L^G(\Omega)$ it will be assumed to be extended by 0 to $\R^n\setminus \Omega$.

Finally, observe that the inclusions
$$
W^{s,G}_0(\Omega)\subset W^{s,G}(\R^n)\subset L^G(\R^n)
$$
imply
$$
L^{G^*}(\Omega)\subset L^{G^*}(\R^n)\subset W^{-s,G^*}(\R^n)\subset W^{-s,G^*}(\Omega),
$$
where $W^{-s,G^*}(\Omega)$ denotes the (topological) dual space of $W^{s,G}_0(\Omega)$.

\subsection{Poincar\'e's  inequality}

A first consequence that we get is Poincar\'e's inequality.

Let us first recall Poincar\'e's inequality for functions in $W^{1,G}_0(\Omega)$ whose proof can be found, for instance, in \cite[Lemma 2.4]{FuchsOsmolovski}.
\begin{equation}\label{eq.Poincare.1}
\int_\Omega G(|u|)\, dx \le A \int_\Omega G(|\nabla u|)\, dx
\end{equation}
for every $u\in W^{1,G}_0(\Omega)$.

\begin{thm}\label{teo.Poincare}
Let $A$ be the optimal constant in Poincar\'e's inequality \eqref{eq.Poincare.1}. Then, given $\delta>0$ there exists $0<s_0<1$ such that
\begin{equation}\label{eq.Poincare.s}
\int_\Omega G(|u|)\, dx \le \left(\frac{A}{c_1} + \delta\right) (1-s)\iint_{\R^n\times\R^n} G\left(\frac{|u(x)-u(y)|}{|x-y|^s}\right) \frac{dxdy}{|x-y|^n},
\end{equation}
for every $s_0\le s<1$ and every $u\in L^G(\Omega)$. The constant $c_1$ is the one given in Proposition \ref{prop.equiv}.
\end{thm}

\begin{proof}
The argument in this proof is taken from \cite{Ponce}.

The proof follows by contradiction. Assume the result is false, therefore there exists a constant $C>\frac{A}{c_1}$, a sequence $s_j\uparrow 1$ and $u_j\in W^{s_j,G}_0(\Omega)$ such that
$$
\Phi_G(u_j) = 1\quad \text{and}\quad (1-s_j)\Phi_{s_j, G}(u_j)\le \frac{1}{C}.
$$

Now, from Theorem \ref{main2}, passing to a subsequence if necessary, we have the existence of a function $u\in W^{1,\tilde G}(\R^n)$ such that $\Phi_G(u-u_j) \to 0$ and $u_j\to u$ a.e. in $\R^n$. Hence, $u\in W^{1,\tilde G}_0(\Omega)$, $\Phi_G(u)=1$ and, again by Theorem \ref{main2} and Proposition \ref{prop.equiv},
$$
c_1\Phi_G(|\nabla u|)\le \Phi_{\tilde G}(|\nabla u|)\le \liminf_{j\to\infty} (1-s_j)\Phi_{s_j,G}(u_j)\le \frac{1}{C}.
$$
This last inequality contradicts the optimality of the constant $A$ in \eqref{eq.Poincare.1}.
\end{proof}

We immediately obtain the next corollary.
\begin{cor}
Let $\Omega\subset \R^n$ be open and bounded. Then there exists a constant $C>0$ depending on $n$, $G$ and $\Omega$ such that
$$
\int_\Omega G(|u|)\, dx \le C (1-s)\iint_{\R^n\times\R^n} G\left(\frac{|u(x)-u(y)|}{|x-y|^s}\right) \frac{dxdy}{|x-y|^n},
$$
for every $0<s<1$ and $u\in L^G(\Omega)$.
\end{cor}

\begin{proof}
To conclude the Corollary, we only need to prove Poincar\'e's inequality for any fixed $0<s<1$. To this end, let $d=\diam(\Omega)$. Thus, from Lemma \ref{lema.iterado},
\begin{align*}
\iint_{\R^n\times\R^n} G\left(\frac{|u(x)-u(y)|}{|x-y|^s}\right) \frac{dxdy}{|x-y|^n}&\ge \int_{\Omega} \int_{|x-y|\ge d+1} G\left(\frac{|u(x)|}{|x-y|^s}\right) \frac{dxdy}{|x-y|^n}\\
&\ge \int_\Omega G(|u(x)|)\left(\int_{|x-y|\ge d+1}\frac{dy}{|x-y|^{n+2qs}}\right) dx\\
&= \frac{n\omega_n}{2qs (d+1)^{2qs}} \int_\Omega G(|u(x)|)\, dx.
\end{align*}
Combining this inequality with Theorem \ref{teo.Poincare} the conclusion of the Corollary follows.
\end{proof}

\subsection{$\Gamma-$convergence}

Let us recall the definition of $\Gamma-$convergence.

\begin{defn}
Let $X$ be a metric space and $F,F_j\colon X  \to \bar \R$. We say that $F_j$  $\Gamma-$converges to $F$ if for every $u\in X$ the following conditions are valid.

\begin{itemize}
\item[(i)] (lim inf inequality) For every sequence $\{u_j\}_{j\in\N}\subset X$ such that $u_j \to u$ in $X$, 
$$
F(u)\leq \liminf_{j\to\infty} F_j(u_j).
$$

\item[(ii)] (lim sup inequality). For every $u\in X$, there is a sequence $\{u_j\}_{j\in\N}\subset X$ converging to $u$ such that 
$$
F(u)\geq  \limsup_{j\to\infty} F_j(u_j).
$$
\end{itemize}
The functional $F$ is called the $\Gamma-$limit of the sequence $\{F_j\}_{j\in\N}$ and it is denoted by $F_j \stackrel{\Gamma}{\to} F$ and 
$$
	F=\glim_{j\to\infty} F_j.
$$
\end{defn}

\begin{rem}
In the case where the functions are indexed by a continuous parameter, $\{F_\ve\}_{\ve>0}$, we say that 
$$
F=\glim_{\ve\downarrow 0} F_\ve,
$$
if and only if for every sequence $\ve_j\downarrow 0$, it follows that $F_{\ve_j}\stackrel{\Gamma}{\to} F$.
\end{rem}

\medskip

Now, let us fix $\Omega\subset \R^n$ open, and an Orlicz function $G$.

For any $0<s<1$, we define the functional $\J_s\colon L^G(\Omega)\to \bar \R$ by
\begin{align*}
	\J_s(u)=
	\begin{cases}
	(1-s)\Phi_{s,G}(u) &\qquad \text{ if } u\in W_0^{s,G}(\Omega)\\
	+\infty &\qquad \text{ otherwise},
	\end{cases}
\end{align*}
and the limit functional $\J\colon L^G(\Omega)\to \bar \R$ 
\begin{align*}
\J(u)=
\begin{cases}
\Phi_{\tilde G}(|\nabla u|) &\text{ if } u\in W_0^{1,\tilde G}(\Omega) \\
+\infty &\text{ otherwise}.
\end{cases}
\end{align*}

\begin{thm} \label{teo.gamma.conv}
With the previous notation we have that
$$
\J= \glim_{s\uparrow 1} \J_s.
$$
\end{thm}
The proof of Theorem \ref{teo.gamma.conv} is a direct consequence of our previous results. Indeed, the limsup inequality follows just by choosing the constant sequence, whilst the liminf inequality follows from Theorem \ref{main2}.

The main feature of the $\Gamma-$convergence is that it implies the convergence of minima. 
\begin{thm} \label{teo.Gamma}
Let $(X,d)$ be a metric space and let $F,F_j\colon X  \to \bar \R$, $j\in\N$, be such that $F_j$  $\Gamma-$converges to $F$. Assume that for each $j\in\N$ there exist $u_j\in X$ such that $F_j(u_j)=\inf_{X} F_j$ and 	suppose that the sequence $\{u_j\}_{j\in\N}\subset X$ is precompact.

Then every accumulation point of $\{u_j\}_{j\in\N}$ is a minimum of $F$ and
$$
	\inf_{X} F = \lim_{j\to\infty} \inf_{X} F_j.
$$
\end{thm} 

The proof of Theorem \ref{teo.Gamma} is elementary. For a comprehensive study of Gamma convergence ant its properties, see \cite{DalMaso}.

Consider now $f\in L^{G^*}(\Omega)$ and define the functionals $\F, \F_s$ as 
\begin{equation}\label{FjF}
\F_s(u) := \J_s(u) - \int_\Omega f u\, dx\quad \text{y}\quad \F(u) := \J(u) - \int_\Omega f u\, dx.
\end{equation}

Since $u\mapsto \int_\Omega f u\,dx$ is continuous in $L^G(\Omega)$, Theorem \ref{teo.gamma.conv} implies that $\F_s\stackrel{\Gamma}{\to} \F$. See \cite[Proposition 6.21]{DalMaso}.

Let us apply Theorem \ref{teo.Gamma} to the familty $\F_s$. With this aim, let us verify that, given $0<s_j\uparrow 1$, there exists a sequence $\{u_j\}_{j\in\N}\in L^G(\Omega)$ of minimizers of $\F_{s_j}$ which is precompact in $L^G(\Omega)$.

The proof of the next lemma is standard. We state it for future references and leave the proof to the reader.
\begin{lema}\label{lema1}
Let $0<s<1$, $G$ be a uniformly convex Orlicz function and $f\in L^{G^*}(\Omega)$. Then there exists a unique function $u\in W^{s,G}_0(\Omega)$ such that
$$
\F_s(u)= \inf_{v\in W^{s,G}_0(\Omega)} \F_s(v).
$$
\end{lema}

Now, a simple consequence of Theorem \ref{main2} gives the compactness of the sequence of minima. Again, the details of the proof are left to the readers.
\begin{lema}\label{lema2}
Let $0<s_j\uparrow 1$, and $\Omega\subset \R^n$ be an open bounded subset. Given $j\in\N$, let $u_j\in L^G(\Omega)$ be the minimum of $\F_{s_j}$. Then  $\{u_j\}_{j\in\N}\subset L^G(\Omega)$ is precompact.
\end{lema}

As a corollary of Lemmas \ref{lema1} and \ref{lema2} and Theorem \ref{teo.Gamma} we obtain the following result.
\begin{thm}\label{teo.conv.minimos}
Let $G$ be a uniformly convex Orlicz function, $\Omega\subset\R^n$ be open and bounded and let $u_s\in  L^G(\Omega)$ be the minimum of $\F_s$. Then there exists $u\in L^G(\Omega)$ such that 
$$
u=\lim_{s\uparrow 1} u_s \text{ in } L^G(\Omega) \qquad \text{ and } \qquad \F(u)=\min_{v\in L^G(\Omega)} \F(v).
$$
\end{thm}

\subsection{The fractional $g-$laplacian operator}

Let $G$ be an Orlicz function and $0<s<1$ be a fractional parameter. We define the fractional $g-$laplacian operator as 
\begin{align} \label{vp}
\begin{split}
(-\Delta_g)^s u&:=\text{p.v.} \intr  G'\left(\frac{|u(x)-u(y)|}{|x-y|^s}\right) \frac{u(x)-u(y)}{|u(x)-u(y)|}\frac{dy}{|x-y|^{n+s}} \\
&=\text{p.v.} \intr  g\left(\frac{|u(x)-u(y)|}{|x-y|^s}\right)\frac{u(x)-u(y)}{|u(x)-u(y)|}\frac{dy}{|x-y|^{n+s}}
\end{split}
\end{align}
where p.v. stands for {\em in principal value} and $g=G'$ which is well defined in view of \eqref{H3}.

Let us see that this operator is well defined between $W^{s,G}(\R^n)$ and its dual space $W^{-s,G^*}(\R^n)$.

For that purpose, let us define the approximating operators as
$$
(-\Delta_g)^s_\ve u(x) := \int_{|x-y|\ge \ve}  g\left(\frac{|u(x)-u(y)|}{|x-y|^s}\right)\frac{u(x)-u(y)}{|u(x)-u(y)|}\frac{dy}{|x-y|^{n+s}}.
$$
We have the following lemma.
\begin{lema}\label{Delta.ve}
Under the above notations and assumptions, there holds that if $u\in W^{s,G}(\R^n)$, then $(-\Delta_g)^s_\ve u\in L^{G^*}(\R^n)$.
\end{lema}

\begin{proof}
First, observe that
\begin{align*}
G^*(|(-\Delta_g)^s_\ve u(x)|)&\le G^*\left(\int_{|h|\ge\ve}  g\left(\frac{|u(x+h)-u(x)|}{|h|^s}\right)\frac{dh}{|h|^{n+s}}\right)\\
&\le C_\ve \int_{|h|\ge\ve}  G^*\left(g\left(\frac{|u(x+h)-u(x)|}{|h|^s}\right)\right)\frac{dh}{|h|^{n+s}}\\
&\le \frac{C_\ve (p-1)}{\ve^s} \int_{|h|\ge\ve}  G\left(\frac{|u(x+h)-u(x)|}{|h|^s}\right)\frac{dh}{|h|^n}
\end{align*}
where we have used Jensen's inequality, the $\Delta_2$ condition \eqref{H2} and Lemma \ref{G*g}.

Integrating over $\R^n$ we obtain
$$
\Phi_{G^*}((-\Delta_g)^s_\ve u)\le \frac{C_\ve (p-1)}{\ve^s} \Phi_{s,G}(u).
$$
This completes the proof.
\end{proof}

\begin{rem}
Although it will not be used here, is not difficult to see that the constant $C_\ve$ in the former inequality is bounded independently of $\ve$. So we get the estimate
$$
\Phi_{G^*}((-\Delta_g)^s_\ve u)\le \frac{C}{\ve^s} \Phi_{s,G}(u).
$$
\end{rem}

It remains to see that $\lim_{\ve\downarrow 0} (-\Delta_g)^s_\ve u$ exists in $W^{-s,G^*}(\R^n)$.
\begin{thm}\label{Delta.defi}
Given $u\in W^{s,G}(\R^n)$, the limit $(-\Delta_g)^s u := \lim_{\ve\downarrow 0} (-\Delta_g)^s_\ve u$ exists in $W^{-s,G^*}(\R^n)$. Moreover the following representation formula holds
$$
\langle (-\Delta_g)^s u,v \rangle = \frac12 \iint_{\R^n\times\R^n} g\left(\frac{|u(x)-u(y)|}{|x-y|^s}\right) \frac{u(x)-u(y)}{|u(x)-u(y)|} \frac{v(x)-v(y)}{|x-y|^s}\frac{dx\,dy }{|x-y|^{n}},
$$
for any $v\in W^{s,G}(\R^n)$.
\end{thm}

\begin{proof}
In view of Lemma \ref{Delta.ve}, we have
\begin{align*}
\langle (-\Delta_g)^s_\ve u, v\rangle &= \intr \int_{|x-y|\ge\ve} g\left(\frac{|u(x)-u(y)|}{|x-y|^s}\right)\frac{u(x)-u(y)}{|u(x)-u(y)|} v(x) \frac{dy}{|x-y|^{n+s}}\, dx\\
&= \intr \int_{|x-y|\ge\ve} g\left(\frac{|u(x)-u(y)|}{|x-y|^s}\right)\frac{u(y)-u(x)}{|u(x)-u(y)|} v(y) \frac{dy}{|x-y|^{n+s}}\, dx.
\end{align*}
Therefore,
$$
\langle (-\Delta_g)^s_\ve u, v\rangle =\frac12 \iint_{|x-y|\ge\ve} g\left(\frac{|u(x)-u(y)|}{|x-y|^s}\right) \frac{u(x)-u(y)}{|u(x)-u(y)|} \frac{v(x)-v(y)}{|x-y|^s}\frac{dx\,dy }{|x-y|^{n}}.
$$
In order to pass to the limit we need to check that the integrand is in $L^1(\R^n\times\R^n)$. But, by \eqref{young} and Lemma \ref{G*g} it holds that
\begin{align*}
g\left(\frac{|u(x)-u(y)|}{|x-y|^s}\right) \frac{|v(x)-v(y)|}{|x-y|^s} &\leq 
G^* \left(  g\left(\frac{|u(x)-u(y)|}{|x-y|^s}\right) \right) + G\left( \frac{|v(x)-v(y)|}{|x-y|^s} \right)\\
&\leq (p-1) G\left(\frac{|u(x)-u(y)|}{|x-y|^s} \right) +  G\left( \frac{|v(x)-v(y)|}{|x-y|^s} \right)
\end{align*}
and hence we conclude
$$
\iint_{\R^n\times \R^n} g\left(\frac{|u(x)-u(y)|}{|x-y|^s}\right) \frac{|v(x)-v(y)|}{|x-y|^s}    \frac{dxdy}{|x-y|^n} \leq (p-1) \Phi_{s,G}(u) + \Phi_{s,G}(v).
$$
The result follows.
\end{proof}

Given a bounded open set $\Omega\subset \R^n$ and $f\in L^{G^*}(\Omega)$, using  Theorem \ref{Delta.defi} we can provide for a notion of weak solution for the following Dirichlet type equation
\begin{align} \label{ec.g}
\begin{cases}
(-\Delta_g)^s u= f &\quad \text{ in } \Omega\\
u=0 &\quad \text{ on } \R^n \setminus \Omega.
\end{cases}
\end{align}

\begin{defn}
We say that $u\in W^{s,G}_0(\Omega)$ is a weak solution of \eqref{ec.g} if
\begin{equation} \label{debil}
\langle (-\Delta_g)^s u, v\rangle = \int_\Omega uv \,dx
\end{equation}
for all $v\in W^{s,G}_0(\Omega)$.
\end{defn}

\begin{rem}
Observe that since $C^\infty_c(\Omega)\subset W^{s,G}_0(\Omega)$, it follows that a weak solution of \eqref{ec.g} is a solution in the sense of distributions.
\end{rem}

Now, as expected, we link every solution to \eqref{ec.g} with minimum points of the associated functional \eqref{FjF}.
\begin{thm}\label{equivalencia}
Let $\Omega\subset \R^n$ be open and bounded and let $f\in L^{G^*}(\Omega)$. Then $u\in W^{s,G}_0(\Omega)$ is a weak solution of  \eqref{ec.g} if and only if
$$
\F_s(u) = \inf_{v\in L^G(\Omega)} \F_s(v),
$$
where the functional $\F_s\colon L^G(\Omega)\to \bar\R$ is given by \eqref{FjF}.
\end{thm}
 
\begin{proof}
Let $u\in W^{s,G}_0(\Omega)$ be a weak solution of \eqref{ec.g} and $v\in W^{s,G}_0(\Omega)$ be an arbitrary function.  Then,
$$
\langle (-\Delta_g)^su,u-v\rangle = \int_\Omega f (u-v)\,dx.
$$
Rearranging the terms, applying \eqref{young} we get
\begin{align*}
\langle (-\Delta_g)^su,u\rangle -& \int_\Omega fu\,dx = \langle (-\Delta_g)^su,v\rangle - \int_\Omega fv\,dx\\
&\leq \frac12 \iint_{\R^n\times\R^n} G^*\left(g\left(\frac{|u(x)-u(y)|}{|x-y|^s}\right)\right)\frac{dxdy}{|x-y|^n} + \F_s(v).
\end{align*}
Finally, we observe that
$$
\langle (-\Delta_g)^su,u\rangle = \frac12 \iint_{\R^n\times\R^n} g\left(\frac{|u(x)-u(y)|}{|x-y|^s}\right)\frac{|u(x)-u(y)|}{|x-y|^s}\frac{dxdy}{|x-y|^n}.
$$
So, using \eqref{caso.igualdad} we arrive at $\F_s(u) \le \F_s(v)$.

Now, suppose that $u\in L^G(\Omega)$ is a minimum of $\F_s$. In particular,  $\F_s(u)<\infty$ and then $u\in W^{s,G}_0(\Omega)$.

Fixed $v\in W^{1,p}_0(\Omega)$ we define $\phi\colon \R\to\R$ as
$$
\phi(t) = \F_s(u+tv).
$$
Hence $\phi(0) = \inf_{t\in\R} \phi(t)$, from where $\phi'(0)=0$. It is straightforward to see that that condition implies that $u$ satisfies \eqref{debil}.
\end{proof}

Combining Theorem \ref{equivalencia} and Lemma \ref{lema1} we obtain the following result.

\begin{thm}
Let $0<s<1$, $G$ be a uniformly convex Orlicz function and $\Omega\subset \R^n$ be an open bounded domain. Then, given $f\in L^{G^*}(\Omega)$ there exists a unique weak solution to \eqref{ec.g}.
\end{thm}

Combining this theorem with our $\Gamma-$convergence result, Theorem \ref{teo.conv.minimos}, we arrive at the main result of this subsection.

\begin{thm}
Let $0<s<1$, $G$ be a uniformly convex Orlicz function and $\Omega\subset \R^n$ be an open bounded domain. Let $u_s\in W^{s,G}_0(\Omega)$ be the weak solution to \eqref{ec.g}. Then, there exists a function $u\in W^{1,\tilde G}_0(\Omega)$ such that $u_s\to u$ in $L^G(\Omega)$ and $u$ is the weak solution to
\begin{equation}\label{ec.local}
\begin{cases}
-\Delta_{\tilde g} u = f & \text{in }\Omega\\
u=0 & \text{on }\partial\Omega,
\end{cases}
\end{equation}
where $\Delta_{\tilde g} u = \diver\left(\tilde g(|\nabla u|)\frac{\nabla u}{|\nabla u|}\right)$ and $\tilde g = \tilde G'$.
\end{thm}

\begin{proof}
At this point the only thing that needs to be justified is that the whole sequence is convergent. But this fact follows from the uniqueness of solutions to \eqref{ec.local}.
\end{proof}

\section*{Acknowledgements}

This paper is partially supported by grants UBACyT 20020130100283BA, CONICET PIP 11220150100032CO and ANPCyT PICT 2012-0153. 

J. Fern\'andez Bonder and A.M. Salort are members of CONICET

\bibliographystyle{amsplain}
\bibliography{biblio}

\end{document}